\documentclass[12pt]{amsart}
\usepackage{amsthm}
\usepackage{mathrsfs}
\usepackage{graphicx, psfrag, epstopdf}

\usepackage{color}

\definecolor{bluegray}{rgb}{0.1, 0.1, 0.6}
\def\bg{\color{bluegray}}
\def\bk{\color{black}}

\newtheorem{theo}{\bg Theorem}
\newtheorem{lemma}[theo]{\bg Lemma}
\newtheorem{coro}[theo]{\bg Corollary}

\theoremstyle{definition}
\newtheorem{define}{\bg Definition}%[section]

\newtheorem{question}{\bg Question}

\newcommand{\Q}{{\mathbb Q}}

\newcommand{\T}{{\mathbb T}}
\newcommand{\A}{{\mathbb A}}
\newcommand{\R}{{\mathbb R}}
\newcommand{\N}{{\mathbb N}}

\def\eps{{\epsilon}}

 \newcommand{\Z}{{\mathbb Z}}

\def\id{{\rm id}}

\def\cC{\mathcal C}

\def\N{{\mathbb{N}}}

\usepackage[colorlinks,linkcolor=bluegray,anchorcolor=bluegray,citecolor=bluegray]{hyperref}
\usepackage{hyperref}

\author{Artur Avila and Bassam Fayad}
\title[\bg Non-differentiable irrational curves for $C^1$ twist maps]{\bg Non-differentiable irrational curves 
for $C^1$ twist map}
\begin{document}

\bg \begin{abstract} \color{black} We construct a $C^1$ 
symplectic twist map $g$ of the annulus that has an essential
invariant curve $\Gamma$ such that $\Gamma$ is not differentiable and  $g$ restricted to $\Gamma$ is minimal.
\end{abstract}

\maketitle

We denote $\T=\R/\Z$ the circle and $\A=\T \times \R$ the annulus. An important class of area preserving maps of the annulus are the so called twist maps or maps that deviate the vertical, since this class of maps describes the behavior of area preserving surface diffeomorphisms in the neighborhood of a generic elliptic periodic point.

More precisely, a $C^1$ diffeomorphism $g$ of the annulus that is isotopic to identity is a positive twist map (resp. negative twist map) if, for any given lift $\tilde{g} : \R^2 \to \R^2$ of $g$, and for every $\theta \in \R$, the maps $r \mapsto \tilde{\pi} \circ \tilde{g}(\theta,r)$ is an increasing (resp. decreasing)
diffeomorphisms. A twist map may be positive or negative. Here $\tilde{\pi}$ denotes the lift of the first projection map.

An {\it essential invariant curve} by a diffeomorphism $g$ of the annulus is a homotopically non-trivial simple loop that is invariant by $g$. 

When $f$ is a $C^1$ twist map, it is known from Birkhoff theroy that any invariant essential curve by $g$ is the graph of a Lipschitz map over $\T$. Furthermore, it was proven by M.-C. Arnaud in \cite{mc1} that 
the latter map must be $C^1$ on a $G_\delta$ set of $\T$ of full measure. 

Numerical experiment show that the invariant curves of a smooth twist map are actually more regular than just what the Birkhoff theory predicts. Moreover, in the perturbative setting of quasi-integrable twist maps, KAM theory provides a large measure set of smooth invariant curves. The question of the regularity of the invariant curves of twist maps is thus a natural question that is also related to the study of how the KAM curves disappear as the perturbation of integrable curves becomes large. 

Another natural and related question is that of the regularity of the boundaries of Birkhoff instability zones. A Birkhoff instability zone of a twist map $g$ is an open set of the annulus that is homeomorphic to the annulus, that does not contain any invariant essential curve, and that is maximal for these properties. By Birkhoff theory the boundary of an instability zone is an invariant curve that is a Lipschitz graph. We refer to the nice introduction of \cite{mc2} where many features and questions are discussed about the boundaries of Birkhoff instability zones.

In \cite{mather} the following question was asked. 

\begin{question} (J. Mather, \cite[Problem 3.1.1]{mather}) Does there exist an example of a symplectic $C^r$ 
twist map
with an essential invariant curve that is not $C^1$ and that contains no periodic point?
\end{question}

In \cite{herman_asterisque}, Herman gave an example of a $C^2$ twist map of the annulus that has a $C^1$ invariant curve  on which the dynamic is conjugated to the one of a Denjoy counterexample. By Denjoy theorem on topological conjugacy of $C^2$ circle diffeomorphisms with irrational rotation number, such a curve cannot be $C^2$.

In \cite{mc3}, M.-C. Arnaud gave an example of a $C^2$ twist map $g$ of the annulus that has a an invariant curve $\Gamma$ that is non-differentiable on  which the dynamic is conjugated to the one of a Denjoy counterexample. In addition, she shows that in any $C^1$ neighborhood of $g$, there exist twist maps with Birkhoff instability zones having $\Gamma$ for a boundary, and having the same dynamics as $g$ on $\Gamma$. 

\bigskip

In \cite{mc2}  the following natural question is raised. 
\begin{question} (M.-C. Arnaud)  
Does there exist a regular symplectic twist map of the annulus  that has an essential invariant curve
that is non-differentiable on which the restricted dynamics is
minimal?
\end{question}

In this note we give a positive answer to this question in low regularity. 

\begin{theo} \label{main} There exists a symplectic $C^1$ twist map of the annulus that has a non differentiable essential invariant curve $\Gamma$ such that the restriction of $g$ to $\Gamma$ is minimal. 
\end{theo}

Due to \cite[Theorem 2]{mc2} we can deduce the following

\begin{coro} \label{main2} There exists a symplectic $C^1$ twist map $g$ of the annulus that has a non differentiable essential invariant curve $\Gamma$ such that the restriction of $g$ to $\Gamma$ is minimal and such that $\Gamma$ is at the boundary of an instability zone of $g$. 
\end{coro}

The derivation of Corollary \ref{main2} from Theorem \ref{main} follows from the general result of \cite{mc2}, Theorem 2 asserting that any essential invariant curve of a $C^1$ twist map $g$ that has an irrational rotation number, can be viewed as a boundary dynamics of Birkhoff instability zone of an arbitrarily nearby $C^1$ twist map. This result
 relies on a perturbation argument involving the Hayashi $C^1$-closing lemma. 
\medskip

Due to a construction by Herman in \cite{herman_asterisque}, Theorem \ref{main} can be derived from the following result on circle homeomorphisms.
\begin{theo} \label{maincercle} There exists a non-differentiable orientation preserving minimal homeomorphism of the circle $f$ such that $f+f^{-1}$ is of class $C^1$. 
\end{theo}

\bg \begin{proof}[Proof of Theorem \ref{main}] \bk To prove that Theorem \ref{maincercle} implies Theorem \ref{main}, we observe following \cite{herman_asterisque}'s beautiful trick  that if $\tilde{f}$ is a lift of a circle homeomorphism $f(\cdot)=\cdot+\psi(\cdot)$ and if $\phi(\cdot):= \frac{1}{2} (\tilde{f}(\cdot)+\tilde{f}^{-1}(\cdot))$ is a $C^1$ function then the map 
$$g: \A \to \A : (\theta,r) \mapsto (\theta+r,r+ \phi(\theta+r)),$$
 is a $C^1$ twist map of the annulus that has as invariant curve the graph $\Gamma$ of $\theta \mapsto \psi(\theta)$, and such that the dynamics of $g$ restricted to $\Gamma$ is conjugated (via the first projection) to  $f$. \bg \end{proof} \color{black}

To prove Theorem \ref{maincercle}, we will work with a special class of circle diffeomorphisms that we call $C$-great. 

\begin{define}{\it  Let $f:\R/\Z \to \R/\Z$ be a $C^1$ diffeomorphism.  For $C>1$,
we say that $x$ is $C$-good
if $\sum_{n \geq 0} |Df^n(x)|^{-2}<C$, while $\sup_{n<0}
|Df^n(x)|=\infty$ and $\sum_{n<0} |Df^n(x)|^{-2}=\infty$.

We say that $f$ is $C$-great if the set of $C$-good points is uncountable
and $\sup_x |Df(x)+Df^{-1}(x)|<C$.}
\end{define}

{ Moreover, we say that $x$ is $C$-great if it is accumulated by an uncountable set of $C$-good points. We then say that the pair $(f,x)$ is $C$-great.

It is not hard to see that if $f$ is $C$-great then it has a $C$-great point $x$. Moreover, a $C$-great point is actually accumulated by an uncountable set of $C$-great points. 
 
We will see later that $C$-great diffeomorphisms exist with arbitrary irrational rotation number. 

Let $0<u<v$. For an injective map $f$ defined from $[x-v,x+v]$ to $\R$, we define $\Delta(f,x,u,v)=\frac{v}{u} \frac{f(x+u)-f(x)}{f(x+v)-f(x)}$. If $f$ is a differentiable circle diffeomorphism then it must satisfy for every $x \in \T$ that the limit as $u,v$ go to $0$ of $\Delta(f,x,u,v)$ exists and equals $1$. }

We now see how starting from a $C$-great diffeomorphism we can prove Theorem \ref{maincercle}. The proof is based on an inductive application of the following lemma.  
 
\begin{lemma}
\label{lemma1}
For every $C>1$ there exists $\epsilon_0>0$ with the following property.
Let $(f,x)$ be $C$-great.
Then for every $\epsilon>0$ there exists a
$C^1$-diffeomorphism $h:\R/\Z \to \R/\Z$ such that 
$g=h \circ f \circ h^{-1}$ satisfies :
\begin{itemize}
\item $h$
is $\epsilon$-close to ${\rm Id}$ in the $C^0$ topology,
\item  $g+g^{-1}$ is
$\epsilon$-close to $f+f^{-1}$ in the $C^1$ topology, 
\item {there exists $y\in \T$ such that $|y-x|<\eps$ and $(g,y)$ is $C$-great, and there exists $u,v \leq \eps$ such that $\Delta(g,y,u,v)>1+\eps_0$.}
\end{itemize}

\end{lemma}

\bg \begin{proof}[Proof that Lemma \ref{lemma1} implies Theorem \ref{maincercle}]  \bk We start with a minimal diffeomorphism $f_0$ and $x_0$ such that $(f_0,x_0)$ is $C$-great. We construct inductively  sequences $h_n$ and $x_n$ and $u_n$ and $v_n$, such that $u_n$ and $v_n$ converge fast to $0$ and $x_n$ converges fast to some limit point $\bar{x}$ and $h_n$ converges fast to ${\rm Id}$ in the $C^0$ topology. Hence $h_n \circ \ldots \circ h_1$ converges to a circle hoemomorphism $H$ and $f_n$ converges to $F=H \circ f_0 \circ H^{-1}$, hence $F$ is minimal. By the second property of Lemma \ref{lemma1}, we can ask that  $f_n+f_n^{-1}$ converges fast in the $C^1$-topology to a $C^1$ diffeomorphism, that must be $F+F^{-1}$. Moreover $\Delta(f_n,x_n,u_n,v_n)>1+\eps_0$ and by the fast convergence of $(f_n,x_n)$ to $(f,\bar{x})$ we guarantee that $\Delta(F,\bar{x},u_n,v_n)>1+\eps_0$ for every $n$. Hence $F$ is not differentiable.  
\bg \end{proof} \color{black}

\bg \begin{proof}[Proof that Lemma \ref{lemma1}] \bk 

To motivate what follows, note that if $g=h \circ f \circ h^{-1}$ with
$Dh(x)=1+e(x)$
then in order to have $g+g^{-1}=f+f^{-1}$ one must have $h \circ f+h \circ
f^{-1}=f \circ h+f^{-1} \circ h$ so that
$De(f(x)) Df(x)+\frac {De(f^{-1}(x))}
{Df(f^{-1}(x))}=e(x) (Df(h(x))+\frac {1} {Df(f^{-1}(h(x))})$. 
Assuming $h$ is $C^0$ close to $\id$, we get that
$(e(f(x))-e(x)) Df(x) Df(f^{-1}(x))$ is close to
$e(x)-e(f^{-1}(x))$.

Let $x$ be $C$-good.  Let $b_j=Df(f^j(x))$, $j \in \Z$.  Define $c_j$, $j
\in \Z$ by $c_0=1$ and $c_j b_j b_{j-1}=c_{j-1}$.

We note that $\sum_{j \geq 0} c_j$ is bounded by a constant $C'$
depending only
on $C$, and $\sum_{j \leq 0} c_j=\infty$, $\liminf_{j \leq 0} c_j=0$.

Let us now define $e_j$,
$j \in \Z$ as follows.  Fix $N$ large such that
$c_{-N}$ is small.  Note that $c_N$ is also small since $N$ is large.
We let $e_j=0$ for $j>N$.  For $|j| \leq N$, we define
$e_j$ so that $e_{j+1}-e_j=-c_j$.  Particularly
$-e_0=e_{N+1}-e_0=-\sum_{0 \leq j
\leq N} c_j$, so $e_0<C'$.  We now let
$N'$ be much larger than $N$ such that $c_{-N'}$ is small and such that
$\alpha=\frac {\sum_{-N' \leq j <-N} c_j} {e_{-N}}$ is large.  We set then
$e_{j+1}-e_j=\frac {c_j} {\alpha}$
for $-N' \leq j \leq -N-1$.  Particularly,
$e_{-N}-e_{-N'}=\frac {1} {\alpha} \sum_{-N' \leq j<-N} c_j=e_{-N}$, so
$e_{-N'}=0$.  We now define $e_j=0$ for $j<-N'$.
Note that $e_j \geq 0$ for all $j$.  Note that by construction,
$(e_{j+1}-e_j) b_j b_{j-1}+(e_{j-1}-e_j)$ is always small and it is non-zero
only for $j=-N'$, $j=-N$ and $j=N+1$.

Choose a small interval $I_0$
centered around $x$, and let
$I_j=f^j(I_0)$, $j \in \Z$.  We may assume that $I_j \cap I_{j'}=\emptyset$
if $0<|j-j'| \leq 2 N'$.  Then for every $\delta>0$ we can
define a diffeomorphism $h':\R/\Z \to \R/\Z$, which is the identity outside
$\cup_{-N'+1 \leq j \leq N+1} I_j$, such that
$\sup_y |Dh'(y)-1|<\delta$, and such that letting
$f'=h' \circ f \circ h'^{-1}$ we have $f'(y)=f^{j+1}(x)+b_j(y-f^j(x))$
whenever $y$ is near $f^j(x)$ and $-N' \leq j \leq N$.

Let us now select an interval $I'_0 \subset I_0$ centered around $x$ such
that letting $I'_j=f'^j(I'_0)$, $j \in \Z$, we have that $f'|I'_j$ is
affine for $-N' \leq j \leq N$.

We will now define a diffeomorphism $h:\R/\Z \to \R/\Z$ which is the
identity outside $\cup_{-N'+1 \leq j \leq N} I'_j$ as follows.
In order to
specify $h$ it is enough to define $Dh$ on $I'_j$ for $-N'+1 \leq j \leq N$.
We let $\phi:\R \to [-\delta,1]$ be a smooth function supported on
$(-1/2,1/2)$,
symmetric around $0$ and such that $\phi(0)=1$ and $\int \phi(x) dx=0$.  We
then let $Dh=1+e_j
\phi \circ A_j$ where $A_j:I'_j \to [-1/2,1/2]$ is an affine
homeomorphism.

Let $g=h \circ f' \circ h^{-1}$.  Note that $F'=f'+f'^{-1}$ is $C^1$ close to
$F=f+f^{-1}$.  Let us show that $G=g+g^{-1}$ is $C^1$ close to $F'$. 
Note that $G=F'$ in the complement of $\cup_{-N' \leq j \leq N} I'_j$,
so it is enough to show that $DG-DF'$ is small in each $I'_j$, $-N' \leq j
\leq N$.  Indeed for $y \in I'_j$, letting $\kappa=\phi \circ A_j(y)$,
we have $DG(h(y))-DF'(h(y))=\frac {\kappa} {(1+e_j \kappa) b_{j-1}}
((e_{j+1}-e_j) b_j b_{j-1}+e_{j-1}-e_j)$, which is small since the term
$\frac {\kappa} {(1+e_j \kappa) b_{j-1}}$ is bounded (as $-\delta \leq
\kappa \leq 1$, $b_{j-1}>C^{-1}$ and $1+e_j \kappa \geq 1-e_j \delta \geq
1/2$ provided $\delta$ is chosen sufficiently small).

Moreover, we have $Dg(x)-Df(x)=\frac {(1+e_1) b_0} {1+e_0}-b_0=\frac {-b_0}
{1+e_0}<-\epsilon_0$ with $\epsilon_0=\frac{1}{C (1+C')}$. {Since $\frac{g(x+u)-g(x)}{u} \sim Dg(x)$ for $u \ll |I_0|$, while $\frac{g(x+v)-g(x)}{v} \sim Df(x)$ for $1\gg u \gg |I_0|$, this allows to exhibit $u$ and $v$ such that $\Delta(g,x,u,v) \geq c \eps_0$ (recall that   $|\frac{g(x+v)-g(x)}{v}|$ is bounded by the {\it a priori} Lipschitz Birkhoff control).

To conclude, we must show that there exists arbitrary close to $x$ a point $y$ such that  $(g,y)$ is $C$-great. It suffices for this to show that in any interval $J$ around $x$ there is an uncountable set of $C$-good points for $g$. By definition of $x$, we know that the latter is true for $f$. } Note that if $y$ is $C$-good for $f$ then $y'=h(h'(y))$ is
$C$-good for $g$ if $\sum_{n \geq 0} |Dg^n(y')|^{-2}<C$.
Fix some $\Lambda>0$ much larger than $\sup e_j$.
Notice that for small $\lambda>0$, we can choose $m>0$ such
that there is an uncountable set $K' \subset J$ of $y$ such that $\sum_{n \geq 0}
|Df^n(y)|^{-2}<C-\lambda$ and $\sum_{n \geq m}
|Df^n(y)|^{-2}<\lambda/\Lambda$.  Now, notice that $h \circ h'=\id$ except
in $\cup_{-N'+1 \leq j \leq N+1} I_j$, and the derivative of
$h \circ h'$ is
bounded by $(1+\delta) (1+\sup e_j)$.  In particular, if $y \in K'$ then
$h(h'(y))$ will be
$C$-good for $g$ provided
$g^n(y) \notin \cup_{-N'+1 \leq j \leq N+1} I_j$ for
$0 \leq n \leq m$.  For any given $y$ which is not in the $f$-orbit of $x$
(a countable set), this happens provided the size of $I_0$ is chosen
sufficiently small, so for small enough $I_0$ we must have uncountably many
$y \in K' \subset J$ satisfying this condition.
\bg \end{proof} \color{black} 

\bigskip 

\begin{lemma}

For every $\alpha \in \R \setminus \Q$, there exist $C>1$ and
$f$ minimal with rotation number $\alpha$ such that $f$ is $C$-great.

\end{lemma}

\bg \begin{proof}   \bk

Given an interval $I=[a,b]$, let $l(I)=a+\frac {3} {8}(b-a)$ and
$r(I)=a+\frac {5} {8} (b-a)$.

Our construction will depend on a sequence
$m_n \in \N$, $n \geq 0$, such that $m_0$ is large and
$m_{n+1}$ is much larger than $m_n$.

Let $\Omega$ be the set of all finite sequences
$\omega=(\omega_1,...,\omega_n)$ of $l$'s and $r$'s and length $|\omega|=n
\geq 0$.  For $\omega \in \Omega$, we define intervals $I_\omega$
inductively as follows.  Let $n=|\omega|$.  If $n=0$ we let $I_\omega$ be
the interval of length $2^{-m_0}$ centered on $0$.  If $n \geq 1$, let
$\omega'$ consist of $\omega$ stripped of its last digit, and let $I_\omega$
be the interval of length $2^{-m_n}$ centered on $t(I_{\omega'})$ where $t
\in \{l,r\}$ is the last digit of $\omega$.

Let $\phi:\R \to [0,1]$ be a smooth function symmetric around $0$,
supported on $[-1/2,1/2]$ and
such that $\phi|[-1/4,1/4]=1$.

Let $A^\omega:I_\omega \to [-1/2,1/2]$ be an affine homeomorphism.
Let us now define functions $\phi^\omega:\R/\Z \to \R$ on
the intervals $I_\omega+k \alpha$, $1 \leq k \leq 2 n-1$, so that
$\phi^\omega(x+k \alpha)=(n-|n-k|) \phi(A^\omega(x))$.

Note that by selecting $m_n$ sufficiently large, we may assume that
$I_\omega+k \alpha$ does not intersect $I_{\omega'}$ if
$|\omega|=|\omega'|=n$ and $0 \leq |k|
\leq 2^{2^n}$ (low recurrence).

In particular, the supports of the $\phi^\omega(x+\alpha)-\phi^\omega(x)$
do not intersect for $\omega$
of a fixed length (here we use $m_n$ large).  We let
$\phi_n=\sum_{|\omega|=n} \phi_\omega$.  Note that
$|\phi_n(x+\alpha)-\phi_n(x)| \leq 1$.  Note that low recurrence implies
that $\phi_n(x-m \alpha)=0$ if $m \geq 0$ and $2n-1+m \leq 2^{2^n}$.

Finally, define a non-decreasing sequence
$\Phi_N:\R/\Z \to \R$ by $\Phi_N=\sum_{0 \leq n \leq N}
\frac {1} {(n+1)^{4/3}} \phi_n$.  Note that if $x \in K=\cap_n
\cup_{|\omega|=n} I_\omega$ and $m \geq 0$,
if $\Phi_N(x-m \alpha) \geq \frac {n (n+1)} {2}$ then
$\phi_{n'}(x-m\alpha)>0$ for some $n \leq n' \leq N$
(since $\phi_n \leq n$), so that $m>2^{2^n}-2n+1$.
In particular, $\Phi_N(x)=0$.  On the other hand, a direct computation gives
$\Phi_N(x+N \alpha) \geq \frac {1} {10} N^{2/3}$, $N \geq 1$.

Clearly
$\Phi_N(x+\alpha)-\Phi_N(x)$ converges in $C^0$ to some continuous function
$\Theta$.

Note that
$\xi_N=\int e^{\Phi_N(x)} dx$ is close to $1$ (since the $m_n$ are large).
Let $\Psi_N=e^{\Phi_N}/\xi_N$,
and let $h_n:\R/\Z \to \R/\Z$ be such
that $Dh_n=\Psi_n$ and $h_n(0)=0$.  Note that $h_n$ converges fast in the
$C^0$ topology to some homeomorphism $h:\R/\Z \to \R/\Z$.
Then $f_n(x)=h_n(h_n^{-1}(x)+\alpha)$
is a sequence converging in the $C^1$ topology, and converging fast in the
$C^0$ topology (here we use $m_n$ large) to some $f$ satisfying
$f(x)=h(h^{-1}(x)+\alpha)$ and $\ln Df=\Theta
\circ h^{-1}$.  In particular, $f$ is minimal.
Note that all $x \in h(K)$ are $C$-good for some absolute $C$, since $Df^n$
is at least of order $e^{n^{2/3}}$ on $h(K)$.
For each $x$,
from time to time $h^{-1}(x)-n \alpha$ will visit
regions where $\sup_N \Phi_N$ is large, so for $x \in h(K)$,
$Df^{-n}(x)$ will be large (since $\Phi_N(h^{-1}(x))=0$).
Moreover, if $x \in h(K)$,
$\Phi_N(h^{-1}(x)-n \alpha)$ is at most of order $(\ln \ln n)^2$, so
$\sum_{n \geq 1} |Df^{-n}(x)|^{-2}=\infty$.
\bg \end{proof} \color{black}

\bg 
 
\end{document}